\documentclass[reqno,11pt]{amsart}
\pdfoutput1

\usepackage{amssymb,color,hyperref,mathrsfs,stmaryrd}
\usepackage{amsmath}

\usepackage[usenames,dvipsnames]{xcolor}

\usepackage[curve,matrix,arrow]{xy}

\numberwithin{equation}{section}
\newtheorem{theorem}[equation]{Theorem}
\newtheorem{definition}[equation]{Definition}

\newtheorem{Th}{Theorem}
\newtheorem{property}[equation]{}

\theoremstyle{definition}

\newtheorem{eg}[equation]{Example}

\newcommand{\F}{\mathcal{F}}

\renewcommand{\L}{\mathcal{L}}
\newcommand{\M}{\mathcal{M}}
\newcommand{\N}{\mathcal{N}}
\newcommand{\K}{\mathcal{K}}
\newcommand{\D}{\mathbf{D}}
\renewcommand{\H}{\mathcal{H}}

\newcommand{\Syl}{\operatorname{Syl}}

\newcommand{\m}{\mathcal}
\newcommand{\ov}{\overline}

\newcommand{\One}{\operatorname{\mathbf{1}}}
\newcommand{\W}{\mathbf{W}}

\def \<{\langle }
\def \>{\rangle }

\renewcommand{\phi}{\varphi}

\title[Products of partial normal subgroups]{Products of partial normal subgroups}

\author[E.~Henke]{Ellen Henke}
\address{Institute of Mathematics, 
University of Aberdeen, Fraser Noble Building, Aberdeen AB24 3UE, U.K.}
\email{ellen.henke@abdn.ac.uk}

\begin{document}

\begin{abstract}
We show that the product of two partial normal subgroups of a locality (in the sense of Chermak) is again a partial normal subgroup. This generalizes a theorem of Chermak and fits into the context of building a local theory of localities.
\end{abstract}

\maketitle

\section{Introduction}

\textbf{Keywords:} Fusion systems; localities; transporter systems.

\smallskip

\textbf{Subject classification:} 20N99, 20D20

\bigskip

Localities were introduced by Andrew Chermak \cite{Chermak:2013}, in the context of his proof of the existence and uniqueness of centric linking systems. Roughly speaking, localities are group-like structures which are essentially the ``same'' as the transporter systems of Oliver and Ventura \cite{Oliver/Ventura}; see the appendix to \cite{Chermak:2013}. As centric linking systems are special cases of transporter systems, the existence of centric linking systems implies that there is a locality attached to every fusion system. It is work in progress of Chermak to build a local theory of localities similar to the local theory of fusion systems as developed by Aschbacher \cite{Aschbacher:2008}, \cite{Aschbacher:2011}. In fact, it seems often an advantage to work inside of localities, where some group theoretical concepts and constructions can be expressed more naturally than in fusion systems. Thus, one can hope to improve the local theory of fusion systems, once a way of translating between fusion systems and localities is established. The results of this paper can be considered as first evidence that some constructions are easier in the world of localities. We prove that the product of partial normal subgroups of a locality is itself a partial normal subgroup, whereas in fusion systems the product of normal subsystems has only been defined in special cases; see \cite[Theorem~3]{Aschbacher:2011}. It is work in progress of Chermak to show that there is a one-to-one correspondence between the normal subsystems of a saturated fusion system $\F$ and the partial normal subgroups of a linking locality attached to $\F$ in the sense of \cite[Definition~2]{Henke:2015}. This is one reason why our result seems particularly important in the case of linking localities. Another reason is that the concept of a linking locality generalizes properties of localities corresponding to centric linking systems and is thus interesting for studying the homotopy theory of fusion systems; see \cite{Broto/Levi/Oliver:2003a}, \cite{BCGLO1}, \cite{BCGLO2}, \cite{Henke:2015}.  It is however crucial for our proof that we work with arbitrary localities, since our arguments rely heavily on the theory of quotient localities introduced by Chermak \cite{Chermak:2015}, and a quotient of a linking locality is not necessarily a linking locality again. Thus, we feel that the method of our proof gives evidence for the value of studying localities in general rather than restricting attention only to the special case of linking localities. 

\smallskip

To describe the results of this paper in more detail, let $\L$ be a partial group as defined in \cite[Definition~2.1]{Chermak:2013} and \cite[Definition~1.1]{Chermak:2015}. Thus, there is an involutory bijection $\L\rightarrow \L,f\mapsto f^{-1}$, called an ``inversion'', and a multivariable product $\Pi$ which is only defined on certain words in $\L$. Let $\D$ be the domain of the product, i.e. $\D$ is a set of words in $\L$ and $\Pi$ is a map $\D\rightarrow \L$. Following Chermak, we call a non-empty subset $\m{H}$ of $\L$ a \textit{partial subgroup} of $\L$ if $h^{-1}\in \H$ for all $h\in\m{H}$ and $\Pi(v)\in\m{H}$ for all words $v$ in the alphabet $\m{H}$ with $v\in\D$. A partial subgroup $\N$ is called a \textit{partial normal subgroup} if $x^f:=\Pi(f^{-1},x,f)\in\N$ for all $x\in\N$ and $f\in\L$ for which $(f^{-1},x,f)\in\D$.  Given two subsets $\M$ and $\N$ of $\L$, the product $\M\N$ is naturally defined by
$$\M\N=\{\Pi(m,n)\colon m\in\M,\;n\in\N,\;(m,n)\in\D\}.$$
The problem is however to show that this is again a partial normal subgroup if $\M$ and $\N$ are partial normal subgroups. Indeed, as we show in Example~\ref{Counterexample}, this is not true in general if $\L$ is an arbitrary partial group. It is true however in the important case that $(\L,\Delta,S)$ is a locality. Chermak \cite[Theorem~5.1]{Chermak:2015} proved this in a special case and we build on his result to prove the general case. More precisely, we prove the following theorem:

\begin{Th}\label{Main}
Let $(\L,\Delta,S)$ be a locality and let $\M,\N$ be partial normal subgroups of $\L$. Then $\M\N=\N\M$ is a partial normal subgroup of $\L$ and $(\M\N)\cap S=(\M\cap S)(\N\cap S)$. Moreover, for every $g\in \M\N$ there exists $m\in\M$ and $n\in \N$ such that $(m,n)\in\D$, $g=\Pi(m,n)$ and $S_g=S_{(m,n)}$.
\end{Th}

To understand the technical conditions stated in the last sentence of the theorem, we recall from \cite{Chermak:2013} and \cite{Chermak:2015} that 
$$S_g=\{s\in S\colon (g^{-1},s,g)\in\D\mbox{ and }s^g\in S\}$$
for any $g\in \L$. Moreover, for a word $v=(g_1,\dots,g_n)$ in $\L$, $S_v$ is the set of all $s\in S$ such that there exist $x_0,\dots,x_n\in S$ with $s=x_0$, $(g_i^{-1},x_{i-1},g_i)\in\D$ and $x_{i-1}^{g_i}=x_i$ for $i=1,\dots,n$. By Proposition~2.6 and Corollary~2.7 in \cite{Chermak:2015}, the sets $S_g$ and $S_v$ are subgroups of $S$, $S_g\in\Delta$ for any $g\in\L$, and $S_v\in\Delta$ if and only if $v\in\D$. Therefore, the condition $S_g=S_{(m,n)}$ stated in the theorem is crucial for proving that certain products are defined in $\L$. This is particularly important for the proof of our next theorem which concerns products of more than two partial normal subgroups. Given subsets $\N_1,\N_2,\dots,\N_l$ of $\L$ define their product via
$$\N_1\N_2\dots \N_l:=\{\Pi(n_1,n_2,\dots,n_l)\colon (n_1,n_2,\dots,n_l)\in\D,\;n_i\in\N_i\mbox{ for }1\leq i\leq l\}.$$
We prove:

\begin{Th}\label{Main2}
 Let $\N_1,\N_2,\dots,\N_l$ be partial normal subgroups of a locality $(\L,\Delta,S)$. Then $\N_1\N_2\dots\N_l$ is a partial normal subgroup of $\L$. Moreover, the following hold:
\begin{itemize}
 \item [(a)] $\N_1\N_2\dots \N_l=(\N_1\dots \N_k)(\N_{k+1}\dots\N_l)$ for every $1\leq k<l$.
 \item [(b)] $\N_1\N_2\dots \N_l=\N_{1\sigma}\N_{2\sigma}\dots\N_{l\sigma}$ for every permutation $\sigma\in S_l$.
 \item [(c)] For every $g\in \N_1\dots\N_l$ there exists $(n_1,\dots,n_l)\in\D$ with $n_i\in\N_i$ for every $i=1,\dots,l$, $g=\Pi(n_1\dots n_l)$ and $S_g=S_{(n_1,\dots,n_l)}$. 
\end{itemize}
\end{Th}

As already mentioned above, it is work in progress of Andrew Chermak to show that for every fusion system $\F$ and a linking locality $(\L,\Delta,S)$ attached to $\F$ there is a one-to-one correspondence between the normal subsystems of $\F$ and the partial normal subgroups of $\L$. When this work is complete, our results will imply the existence of a product of an arbitrary finite number of normal subsystems of $\F$.

\smallskip

In this text only relatively few demands will be made on understanding the concepts introduced in \cite{Chermak:2013} and \cite{Chermak:2015}. In Section~\ref{localities}, we point the reader to the few general results needed about partial groups, give a concise definition of a locality and review some basic facts about localities. In Section~\ref{Quotients}, we summarize what is needed about partial normal subgroups and quotient localities.

\section{Partial groups and localities}\label{localities}
 
We refer the reader to \cite[Definition~2.1]{Chermak:2013} or \cite[Definition~1.1]{Chermak:2015} for the precise definition of a partial group, and to the elementary properties of partial groups stated in \cite[Lemma~2.2]{Chermak:2013} or \cite[Lemma~1.4]{Chermak:2015}. Adapting Chermak's notation we write $\W(\L)$ for the set of words in a set $\L$, $\emptyset$ of the empty word, and $v_1\circ v_2\circ\dots\circ v_n$ for the concatenation of words $v_1,\dots,v_n\in\W(\L)$.

\smallskip

\textbf{For the remainder of this text let $\L$ be a partial group with product $\Pi\colon \D\rightarrow \L$ defined on the domain $\D\subseteq\W(\L)$.}  

\smallskip

Again following Chermak's notation, we set $\One=\Pi(\emptyset)$. Moreover, given a word $v=(f_1,\dots,f_n)\in\D$, we write $f_1f_2\dots f_n$ for the product $\Pi(v)$. 
Recall the definitions of partial subgroups and partial normal subgroups from the introduction. Note that a partial subgroup of $\L$ is always a partial group itself whose product is the restriction of the product $\Pi$ to $\W(\H)\cap\D$. Observe furthermore that $\L$ forms a group in the usual sense if $\W(\L)=\D$; see \cite[Lemma~1.3]{Chermak:2015}. So it makes sense to call a partial subgroup $\H$ of $\L$ a \textit{subgroup of $\L$} if $\W(\H)\subseteq\D$. In particular, we can talk about \textit{$p$-subgroups of $\L$} meaning subgroups of $\L$ whose order is a power of $p$. 

\smallskip

We will need the Dedekind Lemma \cite[Lemma~1.10]{Chermak:2015} in the following slightly more general form:

\begin{property}[Dedekind Lemma] \label{Dedekind}
 Let $\H$, $\K$, $\mathcal{A}$ be subsets of $\L$ such that $\mathcal{A}$ is a partial subgroups of $\L$ and $\K\subseteq\mathcal{A}$. Then $\mathcal{A}\cap(\H\K)=(\mathcal{A}\cap\H)\K$ and $\mathcal{A}\cap(\K\H)=\K(\mathcal{A}\cap\H)$.
\end{property}

\begin{proof}
Clearly, $(\mathcal{A}\cap\H)\K\subseteq\mathcal{A}\cap (\H\K)$. Taking $h\in\H$ and $k\in\K$ with $(h,k)\in\D$ and $hk\in\mathcal{A}$, we have $(h,k,k^{-1})\in\D$ by \cite[Lemma~1.4(d)]{Chermak:2015} and then $h=h(kk^{-1})=(hk)k^{-1}\in\mathcal{A}$  as $\K\subseteq \mathcal{A}$ and $\mathcal{A}$ is a partial subgroup. Hence, $h\in\mathcal{A}\cap\H$ and $hk\in(\mathcal{A}\cap \H)\K$. The second equation follows similarly.
\end{proof}

Before we continue with more definitions, we illustrate the concepts we mentioned so far with examples. For this purpose we say that two groups $G_1$ and $G_2$ \textit{form an amalgam}, if the set-theoretic intersection $G_1\cap G_2$ is a subgroup of both $G_1$ and $G_2$, and the restriction of the multiplication on $G_1$ to a multiplication on $G_1\cap G_2$ is the same as the restriction of the multiplication on $G_2$ to a multiplication on $G_1\cap G_2$. 

\begin{eg}\label{PartialGroupEx}
Let $G_1$ and $G_2$ be groups which form an amalgam. Set $\L=G_1\cup G_2$ and $\D=\W(G_1)\cup\W(G_2)$. Define a partial product $\Pi\colon\D\rightarrow \L$ by sending $v=(f_1,\dots,f_n)\in\W(G_i)$ to the product $f_1\dots f_n$ in the group $G_i$ for $i=1,2$. Define an inversion $\L\rightarrow \L$ by sending $f\in G_i$ to the inverse of $f$ in the group $G_i$ for $i=1,2$. Then $\L$ with these structures forms a partial group. (For readers familiar with the concept of an objective partial group as introduced in \cite[Definition~2.6]{Chermak:2013} or \cite[Definition~2.1]{Chermak:2015} we mention that, setting $\Delta:=\{G_1,G_2\}$, $(\L,\Delta)$ is an objective partial group if $G_1\cap G_2$ is properly contained in $G_1$ and $G_2$.)

\smallskip

Let $\K$ be a subset of $\L$. Then $\K$ is a partial subgroup of $\L$ if and only if $\K\cap G_i$ is a subgroup of $G_i$ for each $i=1,2$. The subset $\K$ is a subgroup of $\L$ if and only if $\K$ is a subgroup of $G_i$ for some $i=1,2$. Moreover, $\K$ is a partial normal subgroup of $\L$ if and only if $(\K\cap G_i)\unlhd G_i$ for $i=1,2$. 
\end{eg}

We use the construction method introduced in the previous example to show that the product of two partial normal subgroups of a partial group is not in general itself a partial normal subgroup.

\begin{eg}\label{Counterexample}
Let $G_1\cong C_2\times C_4$ and let $G_2$ be a dihedral group of order $16$. Choose $G_1$ and $G_2$ such that $G_1$ and $G_2$ form an amalgam with $G_1\cap G_2\cong C_2\times C_2$ and $\Phi(G_1)=Z(G_2)$. Let $\M$ and $\N$ be the two cyclic subgroups of $G_1$ of order $4$. Form the locality $\L$ as in Example~\ref{PartialGroupEx}. As $G_1$ is abelian, a subgroup $\K$ of $G_1$ is normal in $G_1$ and thus a partial normal subgroup of $\L$ if and only if $\K\cap G_2\unlhd G_2$. As $G_1\cap G_2\cong C_2\times C_2$ and $\M$ and $\N$ are cyclic of order $4$, we have $\M\cap G_2=\N\cap G_2=\Phi(G_1)=Z(G_2)\unlhd G_2$. Thus $\M$ and $\N$ are partial normal subgroups of $\L$. The product $\M\N$ in $\L$ is the same as the product $\M\N$ in $G_1$ and thus equal to $G_1$. However, as $G_2$ does not have a normal fours subgroup, $G_1\cap G_2$ is not normal in $G_2$ and thus $\M\N=G_1$ is not a partial normal subgroup of $\L$.  
\end{eg}

The previous example shows that the concept of a partial group (and even the concept of an objective partial group) is too general for our purposes. Therefore, we will focus on localities. We give a definition of a locality which, in contrast to the definition given by Chermak \cite{Chermak:2013} or \cite{Chermak:2015}, does not require the reader to be familiar with the definition of an objective partial group and can easily seen to be equivalent to Chermak's definition. For any $g\in\L$, $\D(g)$ denotes the set of $x\in\L$ with $(g^{-1},x,g)\in\D$. Thus, $\D(g)$ denotes the set of elements $x\in\L$ for which the conjugation $x^g:=\Pi(g^{-1},x,g)$ is defined. If $g\in\L$ and $X\subseteq \D(g)$ we set $X^g:=\{x^g\colon x\in X\}$. If we write $X^g$ for some $g\in\L$ and some subset $X\subseteq \L$, we will always implicitly mean that $X\subseteq\D(g)$.

\begin{definition}\label{LocalityDefinition}
We say that $(\L,\Delta,S)$ is a \textbf{locality} if the partial group $\L$ is finite as a set, $S$ is a $p$-subgroup of $\L$, $\Delta$ is a non-empty set of subgroups of $S$, and the following conditions hold:
\begin{itemize}
\item[(L1)] $S$ is maximal with respect to inclusion among the $p$-subgroups of $\L$.
\item[(L2)] A word $(f_1,\dots,f_n)\in\W(\L)$ is an element of $\D$ if and only if there exist $P_0,\dots,P_n\in\Delta$ such that 
\begin{itemize}
\item [(*)] $P_{i-1}\subseteq \D(f_i)$ and $P_{i-1}^{f_i}=P_i$.
\end{itemize}
\item[(L3)] For any subgroup $Q$ of $S$, for which there exist $P\in\Delta$ and $g\in\L$ with $P\subseteq \D(g)$ and $P^g\leq Q$, we have $Q\in\Delta$.
\end{itemize}
If $(\L,\Delta,S)$ is a locality and $v=(f_1,\dots,f_n)\in\W(\L)$, then we say that $v\in\D$ via $P_0,\dots,P_n$ (or $v\in\D$ via $P_0$), if $P_0,\dots,P_n\in\Delta$ and (*) holds.
\end{definition}

\textbf{From now on let $(\L,\Delta,S)$ be a locality.}

\smallskip

Note that $P=P^{\One}\leq S$ for all $P\in\Delta$. As $\Delta\neq\emptyset$, property (L3) implies thus $S\in\Delta$. For any $g\in\L$, write $c_g$ for the conjugation map
$$c_g\colon \D(g)\rightarrow \L,\;x\mapsto x^g.$$
Recall the definitions of $S_g$ and $S_v$ from the introduction. Note that $S_g\subseteq\D(g)$. For any subgroup $X$ of $\L$ set
$$N_\L(X):=\{f\in\L\colon X\subseteq\D(f),\;X^f=X\}.$$

\begin{property}[Important properties of localities]\label{LocalitiesProp}
The following hold:
\begin{itemize}
\item [(a)] $N_\L(P)$ is a subgroup of $\L$ for each $P\in\Delta$.
\item [(b)] Let $P\in\Delta$ and $g\in\L$ with $P\subseteq S_g$. Then $Q:=P^g\in\Delta$ (so in particular $Q$ is a subgroup of $S$). Moreover, $N_\L(P)\subseteq \D(g)$ and 
$$c_g\colon N_\L(P)\rightarrow N_\L(Q)$$
is an isomorphism of groups.
\item [(c)] Let $w=(g_1,\dots,g_n)\in\D$ via $(X_0,\dots,X_n)$. Then 
$$c_{g_1}\circ \dots c_{g_n}=c_{\Pi(w)}$$
 is a group isomorphism $N_\L(X_0)\rightarrow N_\L(X_n)$.
\item [(d)] For every $g\in\L$, $S_g\in\Delta$. In particular, $S_g$ is a subgroup of $S$.
\item [(e)] For any $w\in\W(\L)$, $S_w$ is a subgroup of $S_{\Pi(w)}$, and $S_w\in\Delta$ if and only if $w\in\D$.
\end{itemize}
\end{property}

\begin{proof}
Properties (a)-(c) correspond to the statements in \cite[Lemma~2.3]{Chermak:2015} except for the fact stated in (b) that $Q\in\Delta$. This is however true by \cite[Proposition~2.6(c)]{Chermak:2015}.  Property (d) is true by \cite[Proposition~2.6(a)]{Chermak:2015} and property (e) is stated in \cite[Corollary~2.7]{Chermak:2015}.
\end{proof}

\section{Partial normal subgroups and quotient localities}\label{Quotients}

In this section we continue to assume that $(\L,\Delta,S)$ is a locality. The following theorem is a special case of Theorem~\ref{Main} and will be used to prove the more general theorem.

\begin{theorem}\label{ImageSituation}
 Let $\M$, $\N$ be partial normal subgroups of $\L$ such that $\M\cap\N=1$. Then $\M\N=\N\M$ is a partial normal subgroup of $\L$. Moreover, for any $f\in\M\N$ there exists $m\in\M$ and $n\in\N$ such that $(m,n)\in\D$, $f=mn$ and $S_f=S_{(m,n)}$.
\end{theorem}

\begin{proof}
As $\M\cap\N\subseteq S$, it follows from \cite[Lemma~5.3]{Chermak:2015} that $\M$ normalizes $\N\cap S$ and $\N$ normalizes $\M\cap S$. So by \cite[Theorem~5.1]{Chermak:2015}, $\M\N=\N\M$ is a partial normal subgroup of $\L$. Moreover, by \cite[Lemma~5.2]{Chermak:2015}, for any $f\in\M\N$ there exist $m\in\M$ and $n\in\N$ such that $(m,n)\in\D$, $f=mn$ and $S_f=S_{(m,n)}$.
\end{proof}

To deduce Theorem~\ref{Main} from Theorem~\ref{ImageSituation}, we need the theory of quotient localities developed in \cite{Chermak:2015}; see also Sections~3 and 4 in \cite{Chermak:2013}. For the convenience of the reader we quickly summarize this theory here. After that we state some more specialized lemmas needed in our proof. 

\smallskip

\textbf{Throughout let $\K$ be a partial normal subgroup of $\L$ and $T=S\cap\K$.}

\smallskip

\begin{property}\label{T}
\begin{itemize}
\item [(a)] $T$ is strongly closed in $(\L,\Delta,S)$, that is $t^g\in T$ for every $g\in\L$ and every $t\in T\cap S_g$. In particular, $T^g=T$ for any $g\in\L$ with $T\subseteq S_g$.
\item [(b)] $T$ is maximal in the poset of $p$-subgroups of $\N$.
\end{itemize}
\end{property}

\begin{proof}
Let $g\in\L$ and $t\in T\cap S_g$. Then $t^g\in S$ and, as $\N$ is a partial normal subgroup, $t^g\in\N$. Hence, $t^g\in S\cap\N=T$. This proves (a). Property (b) is proved in \cite[Lemma~3.1(c)]{Chermak:2015}.
\end{proof}

We write $\uparrow_\K$ for the relation $\uparrow$ introduced in \cite[Definition~3.6]{Chermak:2015}, but with the partial normal subgroup $\N$ replaced by $\K$. Thus $\uparrow_\K$ is a relation on the set $\L\circ\Delta$ of pairs $(f,P)\in\L\times\Delta$ with $P\leq S_f$. For $(f,P),(g,Q)\in\L\circ\Delta$, we have $(f,P)\uparrow_\K(g,Q)$ if there exist $x\in N_\K(P,Q)$ and $y\in N_\K(P^f,Q^g)$ such that $xg=fy$. We say then $(f,P)\uparrow_\K(g,Q)$ via $(x,y)$. One easily sees that $\uparrow_\K$ is reflexive and transitive. Moreover, $(f,P)\uparrow_\K(f,S_f)$ via $(\One,\One)$. An element $f\in\L$ is called \textit{$\uparrow_\K$-maximal} if $(f,S_f)$ is maximal with respect to the relation $\uparrow_\K$ (i.e. if $(f,S_f)\uparrow_\K(g,Q)$ implies $(g,Q)\uparrow_\K(f,S_f)$ for any $(g,Q)\in\L\circ\Delta$). We summarize some important technical properties of the relation $\uparrow_\K$ in the following lemma.

\begin{property}\label{uparrow}
The following hold:
\begin{itemize}
\item [(a)] Every element of $N_\L(S)$ is $\uparrow_\K$-maximal. In particular, every element of $S$ is $\uparrow_\K$-maximal.
\item [(b)] If $f\in\L$ is $\uparrow_\K$-maximal, then $T\leq S_f$.
\item [(c)] (Stellmacher's Splitting Lemma) Let $(x,f)\in\D$ such that $x\in\K$ and $f$ is $\uparrow_\K$-maximal. Then $S_{(x,f)}=S_{xf}$. 
\end{itemize}
\end{property} 

\begin{proof}
Property (a) is \cite[Lemma~3.7(a)]{Chermak:2015}, (b) is \cite[Proposition~3.9]{Chermak:2015} and (c) is \cite[Lemma~3.12]{Chermak:2015}.
\end{proof}

The relation $\uparrow_\K$ is crucial for defining a quotient locality $\L/\K$ somewhat analogously to quotients of groups. A \textit{coset} of $\K$ in $\L$ is of the form $\K f =\{kf\colon k\in\K,\;(k,f)\in\D\}$ for some $f\in\L$. A \textit{maximal coset} of $\K$ is a coset which is maximal with respect to inclusion among the cosets of $\K$ in $\L$. The set of these maximal cosets is denoted by $\L/K$.

\begin{property}\label{MaxCosets}
The following hold:
\begin{itemize}
\item [(a)] $f\in\L$ is $\uparrow_\K$-maximal if and only if $\K f$ is a maximal coset.
\item [(b)] The maximal cosets of $\K$ form a partition of $\L$.
\end{itemize}
\end{property}

\begin{proof}
This is Proposition~3.14(b),(c),(d) in \cite{Chermak:2015}.
\end{proof}

The reader might note that what we call a coset would be more precisely called a right coset. The distinction does however not matter very much, since we are mostly interested in the maximal cosets and, by \cite[Proposition~3.14(a)]{Chermak:2015}, we have $\K f=f\K$ for any $\uparrow_\K$-maximal element $f\in\L$. By Proposition~\ref{MaxCosets}(b), we can define a map
$$\rho:\L\rightarrow \ov{\L}:=\L/\K$$ 
sending $f\in\L$ to the unique maximal coset of $\K$ containing $f$. This should be thought of as a ``quotient map''. We adopt the bar notation similarly as used for groups. Thus, if $X$ is an element or a subset of $\L$, then $\ov{X}$ denotes the image of $X$ under $\rho$. Furthermore, if $X$ is an element or a subset of $\mathbf{W}(\L)$ then $\ov{X}$ denotes the image of $X$ under $\rho^*$, where $\rho^*$ denotes the map $\mathbf{W}(\L)\rightarrow \mathbf{W}(\L)$ with $(f_1,\dots,f_n)\rho^*=(f_1\rho,\dots,f_n\rho)$. In particular,
$$\ov{\D}=\D\rho^*.$$
We note:

\begin{property}\label{-1}
Let $f,g\in\L$ such that $\ov{g}=\ov{f}$ and $f$ is $\uparrow_\K$-maximal. Then $g\in\K f$.
\end{property}

\begin{proof}
By Proposition~\ref{MaxCosets}(a), $\K f$ is a maximal coset, so $\ov{g}=\ov{f}=\K f$ by the definition of $\rho$. Hence, again by the definition of $\rho$, $g\in\K f$. 
\end{proof}

Recall the definition of a homomorphism of a partial groups from \cite[Definition~3.1]{Chermak:2013} and \cite[Definition~1.11]{Chermak:2015}. By \cite[Lemma~3.16]{Chermak:2015}, there is a unique mapping $\ov{\Pi}:\ov{\mathbf{D}}\rightarrow \ov{\L}$ and a unique involutory bijection $\ov{f}\mapsto\ov{f}^{-1}$ such that $\ov{\L}$ with these structures is a partial group and $\rho$ is a homomorphism of partial groups. Since $\rho$ is a homomorphism, we have $\ov{\Pi}(\ov{v})=\ov{\Pi}(v\rho^*)=\Pi(v)\rho=\ov{\Pi(v)}$ for $v\in\D$ and $\ov{f}^{-1}=\ov{f^{-1}}$ by the definition of a homomorphism of partial groups and by \cite[Lemma~1.13]{Chermak:2015}. In particular, $\ov{\One}=\ov{\Pi}(\emptyset)$ is the identity element in $\ov{\L}$. So $\rho$ has kernel $\ker(\rho)=\{f\in\L\colon \ov{f}=\ov{\One}\}=\K\One=\K$. By \cite[Proposition~4.2]{Chermak:2015}, $(\ov{\L},\ov{\Delta},\ov{S})$ is a locality for $\ov{\Delta}:=\{\ov{P}\colon P\in\Delta\}$. We will use this important fact throughout without further reference. We remark:

\begin{property}\label{0}
Let $v=(f_1,\dots,f_n)\in\mathbf{W}(\L)$ such that each $f_i$ is $\uparrow_\K$-maximal and $\ov{v}\in\ov{\D}$. Then $v\in\D$ and $\ov{\Pi}(\ov{v})=\ov{\Pi(v)}$.
\end{property}

\begin{proof}
As $\ov{v}\in\ov{\D}$, there is $u=(g_1,\dots,g_n)\in\D$ such that $\ov{u}=\ov{v}$. Then $\ov{g_i}=\ov{f_i}$ for $i=1,\dots,n$, i.e. $g_i\in\K f_i$ by \ref{-1}. Now by \cite[Proposition~3.14(e)]{Chermak:2015}, $v\in\D$. As seen above, since $\rho$ is a homomorphism of partial groups, $\ov{\Pi}(\ov{v})=\ov{\Pi(v)}$.
\end{proof}

There is a nice correspondence between the partial subgroups of $\L$ containing $\K$ and the partial subgroups of $\ov{\L}$.

\begin{property}\label{PartialSubgroupCorrespondence}
Let $\mathfrak{H}$ be the set of partial subgroups of $\L$ containing $\K$.
\begin{itemize}
\item [(a)] Let $\H\in\mathfrak{H}$. Then the maximal cosets of $\K$ contained in $\H$ form a partition of $\H$.
\item [(b)] Write $\ov{\mathfrak{H}}$ for the set of partial subgroups of $\L$. Then the map $\mathfrak{H}\rightarrow \ov{\mathfrak{H}}$ with $\H\mapsto\ov{\H}$ is well-defined and a bijection. Moreover, for any $\H\in\mathfrak{H}$, we have $\ov{\H}\unlhd \ov{\L}$ if and only if $\H\unlhd \L$.
\end{itemize}
\end{property}

\begin{proof}
Property (a) is \cite[Lemma~3.15]{Chermak:2015}. The map $\rho$ is a homomorphism of partial groups and $(\ov{\L},\ov{\Delta},\ov{S})$ is a locality. From the way $\ov{\D}$ and $\ov{\Delta}$ are defined, it follows that $\rho$ is a projection in the sense of \cite[Definition~4.5]{Chermak:2015}. Hence, property (b) is a reformulation of \cite[Proposition~4.8]{Chermak:2015}.
\end{proof}

\begin{property}\label{Pre2}
 For any subset $X$ of $\L$ and for any partial subgroup $\H$ of $\L$ containing $\K$, $\ov{X}\cap\ov{\H}=\ov{X\cap\H}$.
\end{property}

\begin{proof}
Clearly, $\ov{X\cap\H}\subseteq \ov{X}\cap\ov{\H}$. Let now $x\in X$ such that $\ov{x}\in\ov{\H}$. Then there exists $h\in\H$ such that $\ov{x}=\ov{h}$ and, by \ref{PartialSubgroupCorrespondence}(a), we may choose $h$ such that $\K h$ is a maximal coset. By the definition of $\rho$, this means $x\in \K h\subseteq \H$ and hence $x\in X\cap\H$. Thus $\ov{x}\in\ov{X\cap \H}$, proving $\ov{X}\cap \ov{\H}\subseteq\ov{X\cap\H}$. 
\end{proof}

\begin{property}\label{PreimagesS0}
 Let $R\leq S$. Then $\{f\in\L\colon \ov{f}\in\ov{R}\}=\K R$
\end{property}

\begin{proof}
 Clearly, $\ov{f}\in\ov{R}$ for any $f\in\K R$, as $\K$ is the kernel of $\rho$. Let now $f\in\L$ and $r\in R$ with $\ov{f}=\ov{r}$. As every element of $S$ is $\uparrow_\K$-maximal by \ref{uparrow}(a), it follows from \ref{-1} that $f\in\K r\subseteq \K R$. This proves the assertion.
\end{proof}

\begin{property}\label{PreimagesS}
 Let $T\leq R\leq S$. Then $R=\{s\in S\colon \ov{s}\in\ov{R}\}$ and $N_{\ov{S}}(\ov{R})=\ov{N_S(R)}$.
\end{property}

\begin{proof}
By \ref{PreimagesS0} and the Dedekind Lemma \ref{Dedekind}, we have $\{s\in S\colon \ov{s}\in\ov{R}\}=S\cap (\K R)=(S\cap\K)R=TR=R$. Moreover, for any element $t\in S$ with $\ov{t}\in N_{\ov{S}}(\ov{R})$ and any $r\in R$, we have $\ov{r^t}=\ov{r}^{\ov{t}}\in\ov{R}$, so $r^t\in \{s\in S\colon \ov{s}\in\ov{R}\}=R$. Hence, $N_{\ov{S}}(\ov{R})\leq \ov{N_S(R)}$. As $\rho$ is a homomorphism of partial groups, $\ov{N_S(R)}\subseteq N_{\ov{S}}(\ov{R})$, so the assertion holds.
\end{proof}

\begin{property}\label{Sfbar}
For every $f\in \L$ such that $f$ is $\uparrow_\K$-maximal, we have $\ov{S_f}=\ov{S}_{\ov{f}}$
\end{property}

\begin{proof}
Set $P=S_f$ and $Q=P^f$. As $\rho$ is a homomorphism of partial groups, one easily observes that $\ov{P}\subseteq\ov{S}_{\ov{f}}$. As $(\ov{\L},\ov{\Delta},\ov{S})$ is a locality, $\ov{S}_{\ov{f}}$ is a $p$-group. So assuming the assertion is wrong, there exists $a\in S$ such that $\ov{a}\in N_{\ov{S}_{\ov{f}}}(\ov{P})\backslash\ov{P}$. As $f$ is $\uparrow$-maximal, $T\leq P=S_f$ by \ref{uparrow}(b). Hence, by \ref{PreimagesS} applied with $P$ in the role of $R$, $\ov{a}\in \ov{N_S(P)}$. So by \ref{PreimagesS} now applied with $N_S(P)$ in the role of $R$, $a\in N_S(P)$. Using \ref{LocalitiesProp}(a),(b), we conclude that $A:=P\<a\>$ is a $p$-subgroup of the group $N_\L(P)$ and that $A^f$ is a $p$-subgroup of the group $N_\L(Q)$. As $\ov{A}^{\ov{f}}\subseteq \ov{S}$, we have $\ov{A}^{\ov{f}}\subseteq N_{\ov{S}}(\ov{Q})$. By \ref{T}(a), $T=T^f\leq Q$. Thus, by \ref{PreimagesS}, $\ov{A}^{\ov{f}}\subseteq \ov{N_S(Q)}$. Now \ref{PreimagesS0} yields $A^f\subseteq (\K N_S(Q))\cap N_\L(Q)=N_\K(Q)N_S(Q)$, where the last equality uses the Dedekind Lemma \ref{Dedekind}. Recall that $N_\L(Q)$ is a finite group. Clearly, $N_\K(Q)$ is a normal subgroup of $N_\L(Q)$. It follows from \ref{T}(b) that $T\in\Syl_p(N_\K(Q))$. So $N_S(Q)\in\Syl_p(N_\K(Q)N_S(Q))$ and by Sylow's theorem, there exists $c\in N_\K(Q)$ such that $A^{fc}\leq N_S(Q)$. Then $(f,P)\uparrow_\K (fc,A)$ via $(\One,c)$ contrary to $f$ being $\uparrow_\K$-maximal.
\end{proof}

\begin{property}\label{FindMax}
 Let $f,g\in\L$ such that $\ov{f}=\ov{g}$, $S_f=S_g$ and $f$ is $\uparrow_\K$-maximal. Then $g$ is $\uparrow_\K$-maximal and $\K f=\K g$.
\end{property}

\begin{proof}
As $f$ is $\uparrow_\K$-maximal and $\ov{f}=\ov{g}$, we have $g\in\K f$ by \ref{-1}, i.e. there exists $k\in\K$ with $(k,f)\in\D$ and $g=kf$. By Stellmacher's Splitting Lemma \ref{uparrow}(c), we have $S_g=S_{kf}=S_{(k,f)}$. Hence, $S_f=S_g=S_{(k,f)}$ and thus $k\in N_\L(S_f)$. 
By \ref{LocalitiesProp}(c), $k^{-1}\in N_\L(S_f)$ and $(k^{-1},k,f)\in\D$ as via $S_f$. Hence, $(k^{-1},g)=(k^{-1},kf)\in\D$, $k^{-1}g=k^{-1}(kf)=k^{-1}kf=(k^{-1}k)f=f$ and $S_f^g=S_f^f$. This shows that $(f,S_f)\uparrow_\K (g,S_f)$ via $(k^{-1},\One)$. We conclude that $g$ is $\uparrow_\K$-maximal as $f$ is $\uparrow_\K$-maximal and $\uparrow_\K$ is transitive. By \ref{MaxCosets}, $\K g$ and $\K f$ are both maximal cosets, and the maximal cosets of $\K$ form a partition of $\L$. So it follows that $\K f=\K g$.
\end{proof}

\section{Proof of Theorem~\ref{Main}}

Throughout this section assume the hypothesis of Theorem~\ref{Main}. Set 
$$\K:=\M\cap\N.$$
Observe that $\K$ is a partial normal subgroup of $\L$. As in Section~\ref{Quotients}, let
$$\rho:\L\rightarrow \ov{\L}:=\L/\K$$ 
be the quotient map sending $f\in\L$ to the unique maximal coset of $\K$ containing $f$, and use the bar notation as introduced there. Set
$$T:=\K\cap S.$$

\begin{property}\label{2}
 $\ov{\M}\cap\ov{\N}=1$.
\end{property}

\begin{proof}
 As $\K$ is contained in $\M$ and $\N$, this is a special case of \ref{Pre2}.
\end{proof}

\begin{property}\label{3}
We have $\ov{\M}\,\ov{\N}=\ov{\N}\,\ov{\M}$, and $\ov{\M}\,\ov{\N}$ is a partial normal subgroup of $\ov{\L}$. Moreover, for any $x\in\ov{\M}\,\ov{\N}$, there exist $\ov{m}\in\ov{\M}$ and $\ov{n}\in\ov{\N}$ such that $(\ov{m},\ov{n})\in\ov{\D}$, $x=\ov{m}\,\ov{n}$ and $\ov{S}_x=\ov{S}_{(\ov{m},\ov{n})}$.  
\end{property}

\begin{proof}
By \ref{PartialSubgroupCorrespondence}(b), $\ov{\M}$ and $\ov{\N}$ are partial normal subgroups of $\ov{\L}$. By \ref{2}, $\ov{\M}\cap\ov{\N}=1$. Hence, the assertion follows from Theorem~\ref{ImageSituation}. 
\end{proof}

\begin{property}\label{Split0}
Let $x\in\ov{\M}\,\ov{\N}$. Then there exist $m\in\M$ and $n\in\N$ with $(m,n)\in\D$ such that $m$, $n$ and $mn$ are $\uparrow_\K$-maximal, $x=\ov{m}\,\ov{n}=\ov{mn}$ and $S_{mn}=S_{(m,n)}$. 
\end{property}

\begin{proof}
 By \ref{3}, there exist $\ov{m}\in\ov{\M}$, $\ov{n}\in\ov{\N}$ such that $(\ov{m},\ov{n})\in\ov{\D}$, $x=\ov{m}\,\ov{n}$ and $\ov{S}_{x}=\ov{S}_{(\ov{m},\ov{n})}$. By \ref{PartialSubgroupCorrespondence}(a), we may furthermore choose preimages $m\in\M$ and $n\in\N$ of $\ov{m}$ and $\ov{n}$ such that $m$ and $n$ are $\uparrow_\K$-maximal. Then, by \ref{T}(a) and \ref{uparrow}(b), $m,n\in N_\L(T)$. By \ref{0}, $(m,n)\in\D$ and $\ov{m}\,\ov{n}=\ov{mn}$. It remains to prove that $S_{mn}=S_{(m,n)}$ and that $mn$ is $\uparrow_\K$-maximal. As an intermediate step we prove the following two properties:
\begin{eqnarray}
\label{eq1}S_f&\subseteq& S_{(m,n)}\mbox{ for every }f\in\L\mbox{ with }\ov{f}=x.\\
\label{eq2}S_f&=&S_{(m,n)}\mbox{ for every $\uparrow_\K$-maximal element }f\in\L\mbox{ with }\ov{f}=x.
\end{eqnarray}
For the proof of (\ref{eq1}) and (\ref{eq2}) note first that, by \ref{Sfbar}, $\ov{S}_{\ov{m}}=\ov{S_m}$ and $\ov{S}_{\ov{n}}=\ov{S_n}$ as $m$ and $n$ are $\uparrow_\K$-maximal. Hence
$$\ov{S}_{x}=\ov{S}_{(\ov{m},\ov{n})}=\{\ov{s}\colon \ov{s}\in\ov{S}_{\ov{m}},\;\ov{s}^{\ov{m}}\in\ov{S}_{\ov{n}}\}
=\{\ov{s}\colon s\in S_m,\;\ov{s}^{\ov{m}}\in\ov{S_n}\}.$$
If $s\in S_m$ then, by definition of $S_m$, $(m^{-1},s,m)\in\D$ and $s^m\in S$. Moreover, as $\rho$ is a homomorphism of partial groups, $\ov{s^m}=\ov{s}^{\ov{m}}$. So $\ov{s}^{\ov{m}}\in\ov{S_n}$ is equivalent to $s^m\in S_n$ by \ref{PreimagesS} since $T\leq S_n$. Hence,
$$\ov{S}_{x}=\{\ov{s}:s\in S_m,\;s^m\in S_n\}=\ov{S_{(m,n)}}.$$
As $m,n\in N_\L(T)$, $T\leq S_{(m,n)}$. Clearly, $\ov{S_f}\subseteq \ov{S}_x$ for every $f\in\L$ with $\ov{f}=x$. If such $f$ is in addition $\uparrow_\K$-maximal, then $\ov{S_f}=\ov{S}_x$ and $T\leq S_f$ by \ref{Sfbar} and \ref{uparrow}(b). Now (\ref{eq1}) and (\ref{eq2}) follow from \ref{PreimagesS}. As $\ov{mn}=\ov{m}\,\ov{n}=x$, (\ref{eq1}) yields in particular $S_{mn}\subseteq S_{(m,n)}$ and thus $S_{mn}=S_{(m,n)}$ by \ref{LocalitiesProp}(e). Choosing $f\in\L$ to be $\uparrow_\K$-maximal with $\ov{f}=x$, we obtain from (\ref{eq2}) that $S_f=S_{(m,n)}=S_{mn}$. So $mn$ is $\uparrow_\K$-maximal by \ref{FindMax} completing the proof.
\end{proof}

\begin{property}\label{Split}
 Let $f\in\L$ with $\ov{f}\in\ov{\M}\,\ov{\N}$. Then $f\in\M\N$ and there exist $m\in\M$, $n\in\N$ with $(m,n)\in\D$, $f=mn$ and $S_f=S_{(m,n)}$.
\end{property}

\begin{proof}
 By \ref{Split0}, we can choose $m\in\M$ and $n\in\N$ with $(m,n)\in\D$ such that $mn$ is $\uparrow_\K$-maximal, $\ov{f}=\ov{mn}$ and $S_{mn}=S_{(m,n)}$. Then there  exists $k\in\K$ with $(k,mn)\in\D$ and $f=k(mn)$. As $S_{mn}=S_{(m,n)}$, it follows $S_{(k,mn)}=S_{(k,m,n)}$ and $(k,m,n)\in\D$ by \ref{LocalitiesProp}(e). Hence, $(km,n)\in\D$ and $f=(km)n$ by the axioms of a partial group. As $\K\subseteq\M$, we have $km\in\M$ and so $f=(km)n\in\M\N$. It is now sufficient to show that $S_{(km,n)}=S_f$. As $mn$ is $\uparrow_\K$-maximal, it follows from Stellmacher's Splitting Lemma \ref{uparrow}(c) that $S_f=S_{k(mn)}=S_{(k,mn)}=S_{(k,m,n)}\subseteq S_{(km,n)}$. By \ref{LocalitiesProp}(e), $S_{(km,n)}\subseteq S_{(km)n}=S_f$. So $S_f=S_{(km,n)}$, proving the assertion.
\end{proof}

\begin{proof}[Proof of Theorem~\ref{Main}]
By \ref{3} and \ref{PartialSubgroupCorrespondence}(b), there exists a partial normal subgroup $\H$ of $\L$ containing $\K$ such that $\ov{\H}=\ov{\M}\,\ov{\N}=\ov{\N}\,\ov{\M}$. Then for any $f\in\L$ with $\ov{f}\in\ov{\M}\,\ov{\N}$, there exists $h\in\H$ with $\ov{f}=\ov{h}$. By \ref{PartialSubgroupCorrespondence}(a), we can choose $h$ to by $\uparrow_\K$-maximal. So by \ref{-1}, $f\in\K h\subseteq\H$. This shows $\H=\{f\in\L\colon \ov{f}\in\ov{\M}\,\ov{\N}\}$. 

\smallskip

We need to prove that $\H=\M\N=\N\M$. As the situation is symmetric in $\M$ and $\N$, it is enough to prove that $\H=\M\N$. Since $\rho$ is a homomorphism, for any $m\in\M$ and $n\in\N$ with $(m,n)\in\D$, we have $\ov{mn}=\ov{m}\,\ov{n}\in\ov{\M}\,\ov{\N}$ and thus $mn\in\H$. Hence, $\M\N\subseteq \H$. By \ref{Split}, we have $\H\subseteq \M\N$, so $\H=\M\N$. Moreover, \ref{Split} shows that for every $f\in\M\N$, there exists $m\in\M$ and $n\in\N$ such that $(m,n)\in\D$, $f=mn$ and $S_f=S_{(m,n)}$. So it only remains to prove that $S\cap (\M\N)=(S\cap \M)(S\cap\N)$. Clearly, $(S\cap\M)(S\cap\N)\subseteq S\cap(\M\N)$. Let now $s\in S\cap(\M\N)$. By what we just said, there exists $m\in\M$ and $n\in\N$ with $(m,n)\in\D$, $s=mn$ and $S_s=S_{(m,n)}$. As $S_s=S$ it follows $m,n\in G:=N_\L(S)$. By \ref{LocalitiesProp}(a), $G$ is a subgroup of $\L$. Furthermore, property (L1) in the definition of a locality implies that $S$ is a Sylow $p$-subgroup of $G$. Note that $X:=G\cap\M$ and $Y:=G\cap \N$ are normal subgroups of $G$. Hence, $s=mn\in (XY)\cap S=(X\cap S)(Y\cap S)=(\M\cap S)(\N\cap S)$ completing the proof.
\end{proof}

\section{The Proof of Theorem~\ref{Main2}}

Throughout let $(\L,\Delta,S)$ be a locality with partial normal subgroups $\N_1,\dots,\N_l$. We prove Theorem~\ref{Main2} in a series of lemmas.

\begin{property}\label{Step1}
 Let $1\leq k <l$ such that the products $\N_1\N_2\dots \N_k$ and $\N_{k+1}\N_{k+2}\dots \N_l$ are partial normal subgroups. Suppose furthermore that for any $f\in\N_1\dots \N_k$ and any $g\in\N_{k+1}\dots \N_l$ there exist $u=(n_1,\dots,n_k),v=(n_{k+1},\dots,n_l)\in\D$ such that $n_i\in\N_i$ for $i=1,\dots,l$, $f=\Pi(u)$, $g=\Pi(v)$, $S_f=S_{u}$ and $S_g=S_{v}$. Then $$\N_1\N_2\dots\N_l=(\N_1\N_2\dots \N_k)(\N_{k+1}\N_{k+2}\dots \N_l)$$ 
is a partial normal subgroup of $\L$, and for every $h\in\N_1,\dots,\N_l$ there exists $w=(n_1,\dots,n_l)\in\D$ such that $n_i\in\N_i$ for $i=1,\dots,l$, $h=\Pi(w)$ and $S_h=S_{w}$.
\end{property}

\begin{proof}
By Theorem~\ref{Main}, $(\N_1\N_2\dots \N_k)(\N_{k+1}\N_{k+2}\dots \N_l)$ is a partial normal subgroup of $\L$. If $w=(n_1,\dots,n_l)\in\D$ with $n_i\in\N_i$ for $i=1,\dots,l$, then $u=(n_1,\dots,n_k),v=(n_{k+1},\dots,n_l)\in\D$ and $\Pi(w)=\Pi(\Pi(u),\Pi(v))\in (\N_1\N_2\dots \N_k)(\N_{k+1}\N_{k+2}\dots \N_l)$. This proves $\N_1\N_2\dots \N_l\subseteq (\N_1\N_2\dots \N_k)(\N_{k+1}\N_{k+2}\dots \N_l)$. To prove the converse inclusion, let now 
$$h\in(\N_1\N_2\dots\N_k)(\N_{k+1}\N_{k+2}\dots\N_l).$$ 
Then by Theorem~\ref{Main}, there exists $f\in\N_1\dots \N_k$ and $g\in\N_{k+1}\dots \N_l$ such that $(f,g)\in\D$, $h=fg$ and $S_h=S_{(f,g)}$. By assumption, there exist $u=(n_1,\dots,n_k),v=(n_{k+1},\dots,n_l)\in\D$ such that $n_i\in\N_i$ for $i=1,\dots,l$, $f=\Pi(u)$, $g=\Pi(v)$, $S_f=S_{u}$ and $S_g=S_{v}$. Then $S_h=S_{(f,g)}=S_{u\circ v}$, $u\circ v\in\D$ via $S_h$, and $h=fg=\Pi(\Pi(u),\Pi(v))=\Pi(u\circ v)\in \N_1\N_2\dots \N_l$, proving the assertion.
\end{proof}

\begin{property}\label{Step2}
\begin{itemize}
 \item [(a)] The product $\N_1\N_2\dots \N_l$ is a partial normal subgroup, and for every $f\in\N_1\N_2\dots\N_l$ there exists $w=(n_1,\dots,n_l)\in\D$ such that $n_i\in\N_i$ for $i=1,\dots,l$, $f=\Pi(w)$ and $S_f=S_{w}$.  
 \item [(b)] We have $\N_1\N_2\dots\N_l=(\N_1\N_2\dots \N_k)(\N_{k+1}\N_{k+2}\dots \N_l)$ for every $1\leq k<l$.
\end{itemize}
\end{property}

\begin{proof}
 We prove this by induction on $l$. Clearly, the claim is true for $l=1$. Assume now $l>1$. Then there exists always $1\leq k<l$. For any such $k$, it follows by induction (one time applied with $\N_1,\dots,\N_k$ and one time applied with $\N_{k+1},\dots,\N_l$ in place of $\N_1,\dots,\N_l$) that the hypothesis of \ref{Step1} is fulfilled. So \ref{Step1} implies directly the assertion. 
\end{proof}

\begin{property}\label{Step4}
 Let $\sigma\in S_l$ be a permutation. Then $\N_1\N_2\dots \N_l=\N_{1\sigma}\N_{2\sigma}\dots\N_{l\sigma}$.
\end{property}

\begin{proof}
We may assume that $\sigma=(i,i+1)$ for some $1\leq i<l$, as $S_l$ is generated by transpositions of this form. Note that $\N_1\dots\N_{i-1}$, $\N_i\N_{i+1}$ and $\N_{i+2}\dots \N_l$ are partial normal subgroups by \ref{Step2}(a), where it is understood that $\N_r\dots\N_s=\{\One\}$ if $r>s$. By Theorem~\ref{Main}, we have $\M\N=\N\M$ for any two partial normal subgroups. Using this fact and \ref{Step2}(b) repeatedly, we obtain
\begin{eqnarray*}
 \N_1\N_2\dots\N_l&=&(\N_1\dots\N_{i-1})(\N_i\N_{i+1})(\N_{i+2}\dots\N_l)\\
&=& (\N_1\dots\N_{i-1})(\N_{i+1}\N_{i})(\N_{i+2}\dots\N_l)\\
&=& \N_{1\sigma}\N_{2\sigma}\dots\N_{l\sigma}.
\end{eqnarray*} 
\end{proof}

\begin{proof}[Proof of Theorem~\ref{Main2}]
 It follows from \ref{Step2}(a) that $\N_1\dots\N_l$ is a partial normal subgroup and that (c) holds. Property (a) is \ref{Step2}(b), and property (b) is \ref{Step4}.
\end{proof}

\bibliographystyle{amsplain}
\bibliography{gpfus}


\end{document}